\documentclass{amsart}
\usepackage{amssymb}

\theoremstyle{plain}
\newtheorem{theorem}{Theorem}[section]
\newtheorem{corollary}[theorem]{Corollary}
\newtheorem{lemma}[theorem]{Lemma}
\newtheorem{proposition}[theorem]{Proposition}
\theoremstyle{definition}

\theoremstyle{remark}
\newtheorem{remark}[theorem]{Remark}
\newtheorem{example}[theorem]{Example}

\numberwithin{equation}{section}
\numberwithin{theorem}{section}

\title[Representation of integers by hermitian forms]{An asymptotic formula for representations of integers by indefinite hermitian forms}
\author{Emilio A. Lauret}
\address{FaMAF--CIEM \\ Universidad Nacional de C\'ordoba. X5000HUA  --- C\'ordoba, Argentina}
\email{elauret@famaf.unc.edu.ar}
\subjclass[2010]{Primary 11D45, 11E39; Secondary 58C40}
\keywords{Representation by hermitian forms, hyperbolic lattice point theorem}
\thanks{Supported by Conicet and Secyt-UNC}

\begin{document}
\begin{abstract}
We fix a maximal order $\mathcal O$ in $\mathbb{F}=\mathbb{R},\mathbb{C}$ or $\mathbb{H}$, and an $\mathbb{F}$-hermitian form $Q$ of signature $(n,1)$ with coefficients in $\mathcal O$.
Let $k\in\mathbb{N}$.
By applying a lattice point theorem on $n$-dimensional $\mathbb{F}$-hyperbolic space, we give an asymptotic formula with an error term, as $t\to+\infty$, for the number $N_t(Q,-k)$ of integral solutions $x\in\mathcal O^{n+1}$ of the equation $Q[x]=-k$ satisfying $|x_{n+1}|\leq t$.
\end{abstract}

\maketitle

\section{Introduction}
The representation theory by quadratic forms has a long history.
It starts with the qualitative problem of determining which integers are represented by a given quadratic form.
For example, Fermat, Legendre and Lagrange dealt with the problem of representation as a sum of two, three and four squares respectively.
After that, the quantitative problem was considered by Jacobi, Kloosterman and Liouville among others.
For instance, Jacobi proved that the number of ways a positive integer $k$ can be written as a sum of four squares is
$$
\displaystyle{8\sum_{m|k,\; 4\nmid m} m},
$$
by determining the Fourier coefficients of the theta function associated to the form $x_1^2+\dots+x_4^2$.

These examples are for positive definite quadratic forms.
For indefinite forms, the literature is much less abundant than for the definite case.
In the present work, we study the quantitative problem for quadratic and hermitian forms of signature $(n,1)$.

Let $\mathbb{F}=\mathbb{R},\mathbb{C}$ or $\mathbb{H}$, and let $\mathcal O$ denote a maximal order in $\mathbb{F}$.
Thus $\mathcal O$ is the ring of integer numbers $\mathbb{Z}$ in the real case, the ring of integers of an imaginary quadratic extension of $\mathbb{Q}$ in the complex case (e.g.\ the Gaussian integers $\mathbb{Z}[\sqrt{-1}]$), and, for instance, the ring of Hurwitz integers if  $\mathbb{F}=\mathbb{H}$, though there are many other choices.
We consider in $\mathbb H$ (and then by restriction in $\mathbb{C}$) the canonical involution $\alpha\mapsto\bar\alpha$.
An \emph{$\mathbb{F}$-hermitian matrix} is a square matrix with coefficients in $\mathbb{F}$ that is equal to its own conjugate transpose.

We consider an $\mathbb{F}$-hermitian matrix
\begin{equation}\label{eq:Q}
  Q=\begin{pmatrix}
    A&\\&-a
  \end{pmatrix},
\end{equation}
with $a\in\mathbb{N}$ and $A\in\mathrm{M}_n(\mathcal O)$ a positive definite $\mathbb{F}$-hermitian matrix.
We also denote by $Q$ the induced $\mathbb{F}$-hermitian form of signature $(n,1)$,
$$
Q[x]:= x^*Qx  = A[\hat x] - a\,|x_{n+1}|^{2},\qquad x\in\mathbb{F}^{n+1},
$$
where $\hat x=(x_1,\dots,x_n)^t\in\mathbb{F}^n$.

We consider, for $k\in\mathbb{N}$,  the solution vectors $x\in\mathcal O^{n+1}$ of the equation
\begin{equation}\label{eq:Q[x]=k}
Q[x]=-k.
\end{equation}
Put
\begin{equation}\label{eq:R(Q,k)}
  \mathcal R(Q,-k) =\{x\in\mathcal O^{n+1}:Q[x]=-k\}.
\end{equation}
Since $Q$ is an indefinite form, this set is either empty or infinite.
From now on, we will assume it is not empty, i.e.\ that $-k$ is represented by $Q$.
Consequently, in order to study the quantitative behavior, one needs to impose additional restrictions to the integral solutions of \eqref{eq:Q[x]=k}.

A natural condition is to intersect the set in \eqref{eq:R(Q,k)} with Euclidean balls $\{x\in\mathbb{F}^{n+1}: \|x\|\le s\}$ in $\mathbb{F}^{n+1}$, with $\|x\|^2 := A[\hat x]+a\,|x_{n+1}|^{2}$.
Note that this norm is induced by the positive definite $\mathbb{F}$-hermitian matrix $\left(\begin{smallmatrix}A&\\&a\end{smallmatrix}\right)$.

A related condition consists of requiring that the solutions of \eqref{eq:Q[x]=k} satisfy the bound $|x_{n+1}|\leq t$.
We have, for $x\in\mathcal R(Q,-k)$, that
\begin{equation}\label{eq:relation}
\|x\|\leq s \quad \text{if and only if} \quad |x_{n+1}| \leq \sqrt{\tfrac{s^2+k}{2}}.
\end{equation}
In this paper we establish an asymptotic formula, for large $t$, for
\begin{equation}\label{eq:N_t(Q,k)}
  N_t(Q,-k):=\#\{x\in\mathcal R(Q,-k):|x_{n+1}|\leq t\},
\end{equation}
the number of integral solutions of the equation $Q[x]=-k$ satisfying  the additional restriction $|x_{n+1}|\leq t$.

Our main result, Theorem~\ref{thm:main4}, is the asymptotic formula
\begin{equation}\label{eq:intro-main}
    N_t(Q,-k)=
    \frac{2^{(r-1)(n+1)}}{|d_{\mathcal O}|^{\frac{n+1}2}}
    \frac{a^\rho\,\mathrm{vol}(S^{nr-1})}{2\rho\,|\det Q|^{\frac{r}{2}}}
    \frac{\pi^{\frac{r}{2}}}{\Gamma(\frac{r}{2})}\;
    \delta(Q,-k)\;
    t^{2\rho} + O(t^{\tau}),
\end{equation}
as $t\to+\infty$, for $\mathbb{F}=\mathbb{R}$ or $\mathbb{C}$.
Here $r=\dim_\mathbb{R}(\mathbb{F})$, $\rho=(n+1)\,r/2-1$, $d_{\mathcal O}$ is the discriminant of the quotient field of $\mathcal O$ and $\delta(Q,-k)$ is the local density of the representation \eqref{eq:Q[x]=k} (see~\eqref{eq:densidad}).
The number $\tau$ is defined in \eqref{eq:tau}.
It depends only on the form $Q$, or more precisely, on the first nonzero eigenvalue of the Laplace-Beltrami operator on $\Gamma_Q^0\backslash \mathrm H_\mathbb{F}^n$, where $\Gamma_Q^0$ is a subgroup of the group of unimodular matrices (see~\ref{eq:Gamma_Q^0}) and $\mathrm H_\mathbb{F}^n$ is the $n$-dimensional $\mathbb{F}$-hyperbolic space.
When $\mathbb{F}=\mathbb{R}$ and $n\geq3$, formula \eqref{eq:intro-main} holds with $\tau=n-3/2$ (note that $2\rho=n-1$).

In the particular case when $\mathbb{F}=\mathbb{R}$ and $Q=I_{n,1}=\left(\begin{smallmatrix}I_n&\\&-1\end{smallmatrix}\right)$, the main theorem asserts that the number $N_t(I_{n,1},-k)$ of vectors $x\in\mathbb{Z}^{n+1}$ such that
$$
x_1^2+\dots+x_n^2-x_{n+1}^2=-k
\qquad\text{and}\qquad
|x_{n+1}|\leq t,
$$
satisfies the following asymptotic estimate, as $t\to+\infty$,
$$
N_t(I_{n,1},-k)=
    \frac{\mathrm{vol}(S^{n-1})}{n-1}\,
    \delta(I_{n,1},-k)\,
    t^{n-1} + O(t^{n-3/2}).
$$
This formula is due to J.~Ratcliffe and S.~Tschantz \cite{Ratcliffe-Tschantz}.
The present article was inspired by this work.
They also considered $k<0$ obtaining the same formula for the leading coefficient without an error term.
Although a similar result should hold for $k<0$ in our context, we will not develop this case here since the tools needed are very different.

Concerning the question of counting integral solutions inside Euclidean balls, we define
\begin{equation}\label{eq:intro-N_s(Q,k)}
\widetilde N_s(Q,-k):=\#\{x\in\mathcal R(Q,-k):\|x\|\leq s\}.
\end{equation}
By applying formula \eqref{eq:intro-main} and relation \eqref{eq:relation} (see Remark~\ref{rmk:norms}), one shows that
\begin{equation}\label{eq:intro-main-ball}
\widetilde N_s(Q,-k) = 2^{-\rho}\,C_{Q,k} \; s^{2\rho} + O(s^{\tau}),
\end{equation}
where $C_{Q,k}$ denotes the main coefficient in \eqref{eq:intro-main}.
Related results for the leading term of the counting function of integral points in quite general algebraic varieties lying in balls of increasing radius as  $s\to +\infty$, have been obtained, for example, in \cite{Duke-Rudnick-Sarnak} and \cite{Borovoi-Rudnick}.

In the absence of nonzero exceptional eigenvalues, formulas \eqref{eq:intro-main} and \eqref{eq:intro-main-ball} become respectively
\begin{align*}
           N_t(Q,-k) =&\, C_{Q,k}\; t^{2\rho} + O\left(t^{2\rho(1-\frac{1}{n+1})+\varepsilon}\right),\\
\widetilde N_s(Q,-k) =&\, 2^{-\rho}\, C_{Q,k}\; s^{2\rho} + O\left(s^{2\rho(1-\frac{1}{n+1})+\varepsilon}\right),
\end{align*}
with $\varepsilon=0$ if $\mathbb{F}=\mathbb{R}$ and any $\varepsilon>0$ if $\mathbb{F}=\mathbb{C}$. We recall that $2\rho = r(n+1)-2$, where $r=\dim_\mathbb{R}(\mathbb{F})$.

Our main tool in this paper is the hyperbolic lattice point theorem of P.~Lax and R.~Phillips \cite{Lax-Phillips} in the real case (improved by B.~M.~Levitan \cite{Levitan}), and of R.~Bruggeman, R.~Miatello and N.~Wallach \cite{Bruggeman-Miatello-Wallach} in the general case ($\mathbb{F}=\mathbb{R},\mathbb{C}$ and $\mathbb{H}$) .
For $\mathbb{F}=\mathbb{R}$ we use the best lower bound known for the first eigenvalue of the Laplace-Beltrami operator on $\Gamma_Q^0\backslash\mathrm H_\mathbb{R}^n$ (see Theorem~\ref{thm:lambda_1}), which was obtained in \cite{EGMKloosterman} and in \cite{Cogdell-Li-Piatetski-Shapiro-Sarnak}.

After applying the lattice point theorem we obtain (see Proposition \ref{prop:main3})
\begin{equation}\label{eq:intro-N_t-afterLPT}
N_t(Q,-k)= C_{Q,k}'\; \Big(\sum_{y\in F} |\Gamma_{Q,y}|^{-1}\Big)\;t^{2\rho} + O(t^\tau),
\end{equation}
where $C_{Q,k}'$ is a constant depending only on $Q$ and $k$, $F$ is a representative set of $\mathcal R(Q,-k)$ under the action of the group of unimodular matrices $\Gamma_Q$ (see~\ref{eq:Gamma_Q}) and $\Gamma_{Q,y}$ is the stabilizer of $y$ in $\Gamma_Q$.

When $\mathbb{F}=\mathbb{R}$ or $\mathbb{C}$, by using the theory of C.~L.~Siegel \cite{SiegelIndefinite} for indefinite quadratic forms and its generalization for indefinite complex hermitian forms given in \cite{Raghavan}, formula \eqref{eq:intro-N_t-afterLPT} can be made more explicit, yielding our main formula \eqref{eq:intro-main}. This is carried out by expressing the term $\sum_{y\in F} |\Gamma_{Q,y}|^{-1}$ as a product of the local density $\delta(Q,-k)$ times a constant which depends only on $Q$ and $k$ (see Corollary~\ref{cor:term}).
The derivation of a formula like \eqref{eq:intro-main} from \eqref{eq:intro-N_t-afterLPT} in the quaternionic case, would first require developing for indefinite $\mathbb H$-hermitian forms the classical theory due to Siegel, as done by Raghavan (see \cite{Raghavan}) for indefinite $\mathbb{C}$-hermitian forms.

The paper is organized as follows.
In Section~2 we introduce the geometric context, relating $N_t(Q,-k)$ with the number of elements in an arithmetic subgroup of $\operatorname{Iso}(\mathrm H_\mathbb{F}^n)$ satisfying a geometric condition.
In Section~3, we apply the lattice point theorem to count such lattice points.
Section~4 uses Siegel's theory to compute the main term of the formula.
We conclude with Section~5, which contains the main theorem, together with some examples and remarks.

\section{$\mathbb{F}$-Hyperbolic space}\label{sec:hyperbolic-spaces}
Throughout the paper, given $R$ a ring with identity, we denote by $\mathrm{M}(m,n;R)$ the set of $m\times n$ matrices with coefficients in $R$, just $\mathrm{M}(m,R)$ when $n=m$,
by $\mathrm{GL}(m,\mathbb{F})$ the \emph{general linear group} and by $\mathrm{SL}(m,\mathbb{F})$ its derived group, the \emph{special linear group}.
For a matrix $C\in\mathrm{M}(m,l;\mathbb{F})$ we denote by $C^*$ its conjugate transpose and write $B[C]=C^*BC$, where $B\in\mathrm{M}(m,\mathbb{F})$.
Let $R^m$ denote the right $R$-module $\mathrm{M}(m,1;R)$.
For $\mathbb{F}=\mathbb{R},\mathbb{C}$ and $\mathbb H$, let $\mathrm P\mathbb F^{n}$ be the $n$-dimensional projective space over $\mathbb{F}$, i.e.\ $\mathrm P\mathbb F^{n}=\mathbb{F}^{n+1}\smallsetminus\{0\}/\mathbb{F}^{\times}$ where $\mathbb{F}^\times$ denotes the nonzero elements of $\mathbb{F}$.

We now introduce a model for Riemannian symmetric spaces of real rank one and negative curvature (leaving out the Cayley plane).
These are the real, complex and quaternionic hyperbolic spaces.
For a general reference on this subject see \cite[II.\S10]{Bridson-Haefliger} and \cite[\S19]{Mostow}.

Let $Q$ be the matrix defined as in \eqref{eq:Q}. The set
\begin{equation}\label{eq:H_F^n}
  \mathrm H_\mathbb{F}^n(Q) = \left\{[x]\in \mathrm P\mathbb F^{n} : Q[x]<0\right\}
\end{equation}
will serve as the set of points for the \emph{$Q$-Kleinian model} of $n$-dimensional $\mathbb{F}$-hyperbolic geometry.
Note that the condition $Q[x]<0$ is well defined on the projective space.
The distance function is defined by
\begin{equation}\label{eq:dist}
  \cosh(d([x],[y])) = \frac{|Q(x,y)|}{|Q[x]|^{1/2}\,|Q[y]|^{1/2}},
\end{equation}
where $Q(x,y)=x^*Qy$.

We consider the $\mathbb{F}$-vector space $\mathbb{F}^{n+1}$ endowed with the form $Q(x,y)=x^*Qy$ of type $(n,1)$.
Let $x^\bot= \{u\in\mathbb{F}^{n+1}:Q(x,u)=0\}$ be the \emph{$Q$-orthogonal complement} of $x\in\mathbb{F}^{n+1}$.
If $Q[x]<0$, then the restriction of $Q$ to $x^\bot$ is positive definite.
We identify $x^\bot$ with $T_{[x]}\mathrm H_\mathbb{F}^n(Q)$ using the differential of the natural projection $\mathbb{F}^{n+1}\smallsetminus\{0\}\to\mathrm H_\mathbb{F}^n(Q)$.
We consider the symmetric positive definite $\mathbb{R}$-bilinear form on $T_{[x]}\mathrm H_\mathbb{F}^n(Q)$ given by
\begin{equation}\label{eq:inner_prodct-T_[x]H}
(u,v)= \frac{\mathrm{Re}\left( Q(u,v) \right)}{|Q[x]|},\qquad x,y\in x^\bot.
\end{equation}
In this way $\mathrm H_\mathbb{F}^n(Q)$ is naturally a Riemannian manifolds.
One can check that the metric associated is \eqref{eq:dist}.
Moreover, this metric gives constant curvature $-1$ in the real case, and pinched sectional curvature in the interval $[-4,-1]$ if $\mathbb{F}=\mathbb{C},\mathbb{H}$.

We denote by
$$
\mathrm{U}(Q,\mathbb{F})=\{g\in\mathrm{GL}(n+1,\mathbb{F}):Q[g]=Q\}
$$
the \emph{$Q$-unitary group} and by $\mathrm{SU}(Q,\mathbb{F})=\mathrm{U}(Q,\mathbb{F})\cap \mathrm{SL}(n+1,\mathbb{F})$ the \emph{special $Q$-unitary group}.
For $Q=I_{n,1}$, it is well known the classical notation
$\mathrm{U}(I_{n,1},\mathbb{F}) = \mathrm{O}(n,1)$, $\mathrm{U}(n,1)$, $\mathrm{Sp}(n,1)$ and $\mathrm{SU}(I_{n,1},\mathbb{F}) = \mathrm{SO}(n,1)$, $\mathrm{SU}(n,1)$, $\mathrm{Sp}(n,1)$
for $\mathbb{F}=\mathbb{R},\mathbb{C},\mathbb{H}$ respectively.
The center of the $Q$-unitary group is given by
\begin{equation}\label{eq:Z}
  Z(\mathrm{U}(Q,\mathbb{F}))=
  \begin{cases}
    \{\pm I_{n+1}\} & \text{if $\mathbb{F}=\mathbb{R}$},\\
    \{z I_{n+1}:|z|=1\}\cong S^1 & \text{if $\mathbb{F}=\mathbb{C}$},\\
    \{\pm I_{n+1}\} & \text{if $\mathbb{F}=\mathbb{H}$}.
  \end{cases}
\end{equation}

The group $\mathrm{U}(Q,\mathbb{F})$ acts transitively on $\mathrm H_\mathbb{F}^n(Q)$ by
\begin{equation}\label{eq:vartheta}
  [x]\mapsto g\cdot[x]=[gx].
\end{equation}
Indeed, by the distance formula \eqref{eq:dist}, its elements act by isometries.
Let $\operatorname{Iso}(\mathrm H_\mathbb{F}^n(Q))$ (resp.\ $\operatorname{Iso}^+(\mathrm H_\mathbb{F}^n(Q))$) denote the set of isometries (resp.\ orientation-preserving isometries) of $\mathrm H_\mathbb{F}^n(Q)$.
It is clear that the elements of $Z(\mathrm{U}(Q,\mathbb{F}))$ act as the identity map on $\mathrm H_\mathbb{F}^n(Q)$.
Moreover, we have that
$$
\{1\}
\xrightarrow{\makebox[6mm]{}}
Z(\mathrm{U}(Q,\mathbb{F}))
\xrightarrow{\makebox[6mm]{$\iota$}}
\mathrm{U}(Q,\mathbb{F})
\xrightarrow{\makebox[6mm]{$\vartheta$}}
\operatorname{Iso}(\mathrm H_\mathbb{F}^n(Q))
$$
is an exact sequence, where $\iota$ denotes the inclusion map and $\vartheta$ is defined by \eqref{eq:vartheta}.
Furthermore, the group $\mathrm{PU}(Q,\mathbb{F}):=\mathrm{U}(Q,\mathbb{F})/Z(\mathrm{U}(Q,\mathbb{F}))$ is, up to finite index, the full isometry group $\operatorname{Iso}(\mathrm H_\mathbb{F}^n(Q))$.

\begin{remark}\label{rmk:isometry_groups}
When $\mathbb{F}=\mathbb{R}$ the group $\mathrm{O}(Q):=\mathrm{U}(Q,\mathbb{R})$ has four connected components.
The identity connected component is $$\mathrm{PSO}(Q):=\{g\in\mathrm{SO}(Q):g_{n+1,n+1}>0\},$$ which is isomorphic to $\operatorname{Iso}^+(\mathrm H_\mathbb{R}^n(Q))$.

When $\mathbb{F}=\mathbb{C}$, the group $\operatorname{Iso}(\mathrm H_\mathbb{C}^n(Q))$ is generated by $\mathrm{PU}(Q,\mathbb{C})$ and the conjugation $[x_1,\dots,x_{n+1}]\to [\bar x_1,\dots,\bar x_{n+1}]$. If $\mathbb{F}=\mathbb{H}$, then $\operatorname{Iso}(\mathrm H_\mathbb{H}^n(Q))\cong\mathrm{PU}(Q,\mathbb{H})$ for $n>1$.
\end{remark}

We fix a \emph{maximal order} $\mathcal O$ in $\mathbb{F}$; this means that the subset $\mathcal O\subset\mathbb{F}$ satisfies the following conditions:
  \begin{enumerate}
    \item[(i)] $\mathcal O$ is a lattice in $\mathbb{F}$ (there is an $\mathbb{R}$-basis $v_1,\dots,v_{r}$ of $\mathbb{F}$ such that $\mathcal O=\mathbb{Z} v_1\oplus\dots\oplus \mathbb{Z} v_{r}$);
    \item[(ii)] $\mathcal O$ is a subring of $\mathbb{F}$ containing $1$;
    \item[(iii)] $2\mathrm{Re}(a)\in\mathbb{Z}$ and $a\bar a\in\mathbb{Z}$ for $a\in\mathcal O$;
  \end{enumerate}
and $\mathcal O$ is maximal among all orders (subsets of $\mathbb{F}$   satisfying (i)-(iii)). Here $r:=\dim_\mathbb{R}(\mathbb{F})$.

The elements of $\mathcal O$ will be called \emph{integers}. The ring of integers $\mathbb{Z}$ is the only order in $\mathbb{R}$.
In $\mathbb{C}$, the maximal orders are the rings of integers of the imaginary quadratic extensions $\mathbb{Q}(\sqrt{-D})$ ($D>0$ squarefree) of $\mathbb{Q}$.
More precisely, they are $\mathbb{Z}[\omega]$ where $\omega=(1+\sqrt{-D})/2$ if $-D\equiv 1\pmod4$ and $\omega=\sqrt{-D}$ otherwise.
When $\mathbb{F}=\mathbb{H}$ there are many orders.
As a canonical example in this case, the reader may take
$$\mathcal O=\left\{a+bi+cj+dk\in\mathbb{H}:a,b,c,d\in\mathbb{Z}\text{ or } a,b,c,d\in\mathbb{Z}+\tfrac12\right\},$$
the \emph{Hurwitz integers}.

Let $\Gamma_Q$ be the set of unimodular matrices in $\mathrm{U}(Q,\mathbb{F})$, that is
\begin{equation}\label{eq:Gamma_Q}
\Gamma_Q=\mathrm{U}(Q,\mathbb{F})\cap\mathrm{M}(n+1,\mathcal O).
\end{equation}
This is a discrete subgroup of $\mathrm{U}(Q,\mathbb{F})$ with finite center.
The action of $\Gamma_Q$ on $\mathrm H_\mathbb{F}^n(Q)$ is discontinuous, not free and the quotient $\Gamma_Q\backslash \mathrm H_\mathbb{F}^n(Q)$ is of finite volume and not compact.

On the other hand, the group $\Gamma_Q$ acts by left multiplication on the set $\mathcal R(Q,k)$ given in \eqref{eq:R(Q,k)}.

\begin{lemma}
The set of $\Gamma_Q$-orbits in $\mathcal R(Q,k)$ is finite.
\end{lemma}
\begin{proof}
  This assertion follows by applying \cite[Thm.~6.9]{Borel-Harish-Chandra}.
\end{proof}

From now on, we fix $k\in\mathbb{N}$ such that $Q$ \emph{represents} $-k$, that is, there exists $x\in\mathcal O^{n+1}$ satisfying $Q[x]=-k$.
Let $F$ be a (finite) set of representatives of the $\Gamma_Q$-orbits of $\mathcal R(Q,-k)$.
Let $\Gamma_{Q,y}$ be the stabilizer of $y$ in $\Gamma_Q$, which is finite.
We conclude this section by relating the number $N_t(Q,k)$ defined in \eqref{eq:N_t(Q,k)}, with the cardinality of subsets of lattice points in $\Gamma_Q$.

\begin{proposition}\label{thm:main2}
  For $t>0$, we have
  \begin{equation}\label{eq:main2}
    N_t(Q,-k)= \sum_{y\in F} |\Gamma_{Q,y}|^{-1} \;\#\left\{g\in\Gamma_Q:d([e_{n+1}],g\cdot[y])\leq s\right\},
  \end{equation}
  where $s=\operatorname{arccosh}(a^{1/2}k^{-1/2}\, t)>0$.
\end{proposition}
\begin{proof}
Put $$\mathcal R_t(Q,-k)=\{x\in\mathcal R(Q,-k):|x_{n+1}|\leq t\}.$$
The cardinality of this set is $N_t(Q,-k)$.
Let $x=(x_1,\dots,x_{n+1})^t\in\mathcal R(Q,-k)$. By \eqref{eq:dist},
\begin{equation}\label{eq:cosh(d(e,x))}
\cosh\left(d([e_{n+1}],[x])\right)= {a^{1/2}}{k^{-1/2}}\, |x_{n+1}|.
\end{equation}
Fix $t>0$, thus $\cosh(s)=a^{1/2}k^{-1/2}\,t$. Let $g\in\Gamma_Q$ and $y\in F$ such that $gy=x$, then $g\cdot[y]=[x]$.
Note that \eqref{eq:cosh(d(e,x))} tells that
\begin{equation*}
  |x_{n+1}|\leq t\quad\text{if and only if}\quad d([e_{n+1}],g\cdot[y])\leq s.
\end{equation*}
But the condition on the left ensures that $x\in\mathcal R_t(Q,-k)$.
We conclude that
\begin{equation}\label{eq:R_t(Q,-k)}
\mathcal R_t(Q,-k)=
\bigcup_{y\in F} \{gy:g\in\Gamma_Q\;\text{ and }\;d([e_{n+1}],g\cdot[y])\leq s\}.
\end{equation}
The proposition follows by counting the elements of these sets.
\end{proof}

\section{Lattice point theorem}
In this section we use lattice point theorems to determine the asymptotic distribution, for $t\to+\infty$, of the number of elements in the sets $$\left\{g\in\Gamma_Q:d([e_{n+1}],g\cdot[y])\leq s\right\}.$$

Let
\begin{equation}
G=\mathrm{SU}^0(Q,\mathbb{F}),
\end{equation}
the identity connected component of $\mathrm{SU}(Q,\mathbb{F})$.
We have $G=\mathrm{PSO}(Q)$ for $\mathbb{F}=\mathbb{R}$ (see Remark~\ref{rmk:isometry_groups}).
When $\mathbb{F}=\mathbb{C},\mathbb H$ the group $\mathrm{SU}(Q,\mathbb{F})$ is connected.
Let $\mathfrak g$ denote the Lie algebra $\mathfrak{su}(Q,\mathbb{F})$ of $G$.
We have
$$
\mathfrak g= \big\{X\in\mathrm{M}(n+1,\mathbb{F}): X^*Q+QX=0\; (\text{and $\textrm{Tr}(X)=0$ if $\mathbb{F}=\mathbb{C}$})\big\}
$$

\begin{remark}\label{rmk:cartan-involution}
If $T=\left(\begin{smallmatrix}L&\\ &\sqrt{a}\end{smallmatrix}\right)$, where $L^*L=A$, then $I_{n,1}[T]=T^*I_{n,1}T=Q$.
It follows that the map $g\mapsto Tg T^{-1}$ gives an isomorphism from $\mathrm{U}(Q,\mathbb{F})$ to $\mathrm{U}(n,1;\mathbb{F}):=\mathrm{U}(I_{n,1},\mathbb{F})$ and thus also from $G$ to $\mathrm{SU}^0(n,1;\mathbb{F})$.
The corresponding isomorphism at the Lie algebra level $\mathfrak{su}(Q,\mathbb{F})\to\mathfrak{su}(n,1;\mathbb{F})$ is given by $X\mapsto TXT^{-1}$.
This isomorphism allows us to consider the Cartan involution $\theta$ on $\mathfrak g$ by pulling back the standard Cartan involution $X\mapsto -X^*$ on $\mathfrak{su}(n,1;\mathbb{F})$.
It is easy to check that $\theta(X)= -(T^*T)^{-1}X^*(T^*T)$.
\end{remark}

The group $G$ is a connected semisimple Lie group of real rank one and finite center.
Let $\theta$ be the Cartan involution given in Remark~\ref{rmk:cartan-involution}, with corresponding Cartan decomposition $\mathfrak g=\mathfrak k\oplus\mathfrak p$.
Let $H_0$ be the matrix in $\mathfrak g$ defined as the pull back from $\mathfrak{su}(n,1;\mathbb{F})$ of the matrix
$$
\left(\begin{array}{c|c}\rule{0pt}{20pt}\rule{20pt}{0pt}& \raisebox{8pt}{$e_n$}\\ \hline e_n^*&\end{array}\right).
$$
Let $\mathfrak a=\mathbb{R} H_0$, a maximal abelian subspace of $\mathfrak p$ since $G$ has real rank one.
Let $\mathfrak g=\mathfrak k\oplus \mathfrak a\oplus\mathfrak n$ and $G=NAK$ be the corresponding Iwasawa decomposition of $\mathfrak g$ and $G$, respectively.
Here $K$ is a maximal compact subgroup of $G$, with Lie algebra $\mathfrak k$.
Let $M$ be the centralizer of $A$ in $K$ with Lie algebra $\mathfrak m$.
Set $\zeta=\mathrm{vol}(K/M)$.
Let $2\rho$ denote the sum of the positive roots of $G$, thus $2\rho= (n+1)r-2=$ $n-1$, $2n$, $4n+2$ for $\mathbb{F}=\mathbb{R},\mathbb{C},\mathbb{H}$ respectively (recall that $r=\dim_\mathbb{R}(\mathbb{F})$).

Let $B_K$ denote the Killing form on $\mathfrak g$.
We shall work with the inner product $\langle\rule{3pt}{0ex},\rule{2pt}{0ex}\rangle$ on $\mathfrak g$ defined by
\begin{equation}\label{eq:inner_product_g}
  \langle X,Y\rangle=-\frac{1}{\xi_\mathbb{F}}\,B_K(X,\theta Y),\qquad
  \text{where }\;
  \xi_\mathbb{F}=
\left\{
\begin{smallmatrix}
  2(n-1)&\quad\text{for }\mathbb{F}=\mathbb{R},\\
  4(n+1)&\quad\text{for }\mathbb{F}=\mathbb{C},\\
  8(n+2)&\quad\text{for }\mathbb{F}=\mathbb{H}.
\end{smallmatrix}
\right.
\end{equation}
A simple computation shows that $\langle X,Y\rangle=\frac12\textrm{Tr}(X^*\,Y)$, thus $\langle H_0,H_0\rangle=1$.
We consider on the homogeneous manifold $G/K$, the $G$-invariant Riemannian metric defined by the restriction of $\langle\cdot,\cdot\rangle$ to $\mathfrak p$.

One can check that the action of $G$ on $\mathrm H_\mathbb{F}^n(Q)$ given in \eqref{eq:vartheta} is transitive, the element $[e_{n+1}]$ lies in $\mathrm H_\mathbb{F}^n(Q)$ and its stabilizer subgroup is $K$.
Hence, the map $g\mapsto [ge_{n+1}]$ from $G$ to $\mathrm H_\mathbb{F}^n(Q)$ gives rise to a $G$-equivariant bijection between the symmetric space $G/K$ and $\mathrm H_\mathbb{F}^n(Q)$.
Moreover, it follows by standard arguments that this bijection is already an isometry of Riemannian manifolds.
This gives to $\mathrm H_\mathbb{F}^n(Q)$ the structure of a Riemannian symmetric space.

Let $\Gamma$ be a non-cocompact lattice in $G$.
Set $m_\Gamma=|\Gamma\cap Z(G)|$.
Let $\mathrm{vol}(\Gamma \backslash\mathrm H_\mathbb{F}^n(Q))$ denotes the volume on any fundamental domain in $\mathrm H_\mathbb{F}^n(Q)$ relative to $\Gamma_Q$.
Let $\Delta$ be the Laplace-Beltrami operator on $\Gamma\backslash \mathrm H_\mathbb{F}^n(Q)$.
We identify $-\Delta$ with the Casimir operator $C$ of $G$, with respect to the inner product defined in \eqref{eq:inner_product_g}.
We fix a complete orthonormal set $\{\varphi_j\}$ of real valued eigenfunctions of $C$, with eigenvalues $\lambda_j$ arranged in increasing order and exceptional eigenvalues $0=\lambda_0<\lambda_1\leq\dots\leq\lambda_N<\rho^2$, which we write as $\lambda_j=\rho^2-\nu_j^2$, where $0<\nu_N\leq\nu_{N-1}\leq\dots\leq\nu_1<\rho$.

Now we can state the \emph{hyperbolic lattice point theorem}.
It was proved for the real hyperbolic space by Lax and Phillips \cite{Lax-Phillips}, with an improved error term by Levitan \cite{Levitan}, and generalized by Bruggeman, Miatello and Wallach \cite{Bruggeman-Miatello-Wallach} for any symmetric space of real rank one.

\begin{theorem}\label{thm:LPT_BMW}
In the notation above, for $[x],[y]\in \mathrm H_\mathbb{F}^n(Q)$, we have that
\begin{multline}\label{eq:LPT_BMW}
\#\{g\in\Gamma:d([x],g\cdot [y])\leq s\}=
\frac{2^{1-n} \, m_\Gamma \, \zeta}{2\rho \,\mathrm{vol}(\Gamma\backslash \mathrm H_\mathbb{F}^n(Q))}\, e^{2\rho s}\\
+2^{1-n}m_\Gamma \zeta\;\sum_{j=1}^N\frac{c(\nu_j)}{\nu_j+\rho}\, \varphi_j(x)\varphi_j(y)\, e^{(\rho+\nu_j)s}
+O\left(e^{(2\rho\frac{n}{n+1}+\varepsilon)s}\right)
\end{multline}
as $s\to+\infty$, for $\varepsilon=0$ when $\mathbb{F}=\mathbb{R}$ and for any $\varepsilon>0$ otherwise.
Here $c(\nu)$ is the Harish-Chandra $c$-function.
\end{theorem}

Note that the summation in \eqref{eq:LPT_BMW} can be restricted to the indices $j$ such that $\rho+\nu_j> 2\rho\,\frac{n}{n+1}$ (in the real case we can replace $>$ by $\geq$).
Put
\begin{eqnarray}\label{eq:tau}
\tau&=&
 \begin{cases}
  \rho+\nu_1& \text{if}\quad \rho+\nu_1\geq 2\rho \frac{n}{n+1}\quad \text{and}\quad\mathbb{F}=\mathbb{R},\\[1mm]
  \rho+\nu_1& \text{if}\quad \rho+\nu_1  >  2\rho \frac{n}{n+1}\quad\text{and}\quad\mathbb{F}=\mathbb{C},\mathbb{H},\\[1mm]
  2\rho\frac{n}{n+1}+\varepsilon&\text{otherwise,}
 \end{cases}
\end{eqnarray}
where $\varepsilon$ is zero if $\mathbb{F}=\mathbb{R}$ or any positive value if $\mathbb{F}=\mathbb{C},\mathbb{H}$.
The last case includes the case when there are no exceptional eigenvalues.
With this notation we can rewrite \eqref{eq:LPT_BMW} as
\begin{equation}\label{eq:LPT-unificado}
\#\{g\in\Gamma:d([x],g\cdot [y])\leq s\} =
\frac{2^{1-n} \, m_\Gamma \, \zeta }{2\rho \,\mathrm{vol}(\Gamma\backslash \mathrm H_\mathbb{F}^n(Q))}\, e^{2\rho s}
+O\left(e^{\tau s}\right).
\end{equation}

We see that in \eqref{eq:LPT-unificado} the error term depends on the first nonzero eigenvalue of the Laplace-Beltrami operator on $\Gamma\backslash \mathrm H_\mathbb{F}^n(Q)$.
The following theorem was proved in \cite[Thm.~A]{EGMKloosterman} (see also \cite{Cogdell-Li-Piatetski-Shapiro-Sarnak}).
The notation $\mathrm{PSO}(Q)$ was introduced in Remark~\ref{rmk:isometry_groups}.

\begin{theorem}\label{thm:lambda_1}
  Let $n\geq3$ and let $Q_0$ be a quadratic form with rational coefficients such that $Q_0$ is of signature $(n,1)$ and isotropic over $\mathbb{Q}$.
  For any congruence subgroup $\Gamma<\mathrm{PSO}(Q_0)$, the first nonzero eigenvalue $\lambda_1$ for the Laplace-Beltrami operator on $\Gamma\backslash \mathrm H_\mathbb{R}^n(Q)$ satisfies $\lambda_1\geq (2n-3)/4$.
\end{theorem}

Bounds of this kind are not known in the complex and quaternionic case (see \cite[Cor.~1.4]{Li} for a related result  in the complex case).

\begin{remark}\label{rmk:lambda_1}
Under the assumptions of Theorem~\ref{thm:lambda_1}, the value of $\nu_1$ satisfies
\begin{align*}
  \nu_1&\leq \sqrt{\left(\frac{n-1}2\right)^2 -\frac{2n-3}4}=\frac{n-2}2.
\end{align*}
Moreover, in the worst possible case $\nu_1=\frac{n-2}2$, one has $\rho+\nu_1\geq 2\rho \frac{n}{n+1}$ for all $n\geq3$, which implies that \eqref{eq:LPT-unificado} holds for $\tau=\frac{n-1}2+\frac{n-2}2=n-\frac32$ since $\mathbb{F}=\mathbb{R}$.
\end{remark}

Before applying Theorem~\ref{thm:LPT_BMW} to our problem, we give the value of $\zeta=\mathrm{vol}(K/M)$.
\begin{lemma}\label{lem:zeta}
Under the notation above, one has $\zeta=\mathrm{vol}(S^{nr-1})$.
\end{lemma}
\begin{proof}
It is sufficient to prove the lemma for $Q=I_{n,1}$ since the isomorphism between $G$ and $\mathrm{SU}^0(n,1;\mathbb{F})$ given in Remark~\ref{rmk:cartan-involution} preserves the Killing form and the inner product \eqref{eq:inner_product_g}.

Set $S=\{H\in\mathfrak p:\langle H,H\rangle =1\}$ with the Riemannian metric given by the real inner product $\langle\cdot,\cdot\rangle$ restricted to $\{X\in\mathfrak p: \langle X,H\rangle=0\}\cong T_HS$.
The adjoint representation restricted to $K$ leaves the set $S$ invariant, this action is transitive and the stabilizer subgroup of $H_0\in S$ is $M$.
Then, we have a $K$-equivariant bijection from the manifold $K/M$ to $S$, that is already an isometry, considering the Riemannian metric on $K/M$ given by \eqref{eq:inner_product_g} restricted to $\mathfrak k\cap\mathfrak m^\bot$.

It remains to prove that the Riemannian manifold $S$ is isometric to the $(nr-1)$-dimensional sphere in $\mathbb{R}^{nr}$.
It can be checked that
\begin{equation}\label{eq:p_su(n,1)}
\mathfrak p=\left\{X_v=\left(\begin{array}{c|c}\rule{0pt}{12pt}\rule{12pt}{0pt}& \raisebox{3pt}{$v$}\\ \hline v^*&\end{array}\right) : v\in\mathbb{F}^n \right\}.
\end{equation}
Then $\langle X_v,X_w\rangle = \frac12\mathrm{Tr}(X_v^*\, X_w)= \mathrm{Re} (\sum_{l=1}^n \bar v_l w_l)$.
Identifying $\mathfrak p$ with $\mathbb{F}^n\cong\mathbb{R}^{rn}$ in the obvious way, the inner product $\langle\cdot,\cdot\rangle$ on $\mathfrak p$ coincides with the canonical inner product on the (real) vector space $\mathbb{F}^{n}$ of dimension $nr$.
Hence $S$ is the $(nr-1)$-dimensional sphere of radius one in the (standard) Euclidean space $\mathbb{R}^{nr}$.
\end{proof}

For a more detailed description of the realizations of the symmetric spaces $G/K$ and $K/M$, we refer to
\cite[\S10, Ch.~XI]{Kobayashi-Nomizu}.

Let
\begin{equation}\label{eq:Gamma_Q^0}
\Gamma_Q^0\;=\;G\cap\Gamma_Q \;=\; \mathrm{SU}^0(Q,\mathbb{F})\cap\mathrm{M}(n+1,\mathcal O).
\end{equation}
This is a finite index subgroup of $\Gamma_Q$ (see~\eqref{eq:Gamma_Q}).
Let us denote by $w$ the number of units in $\mathcal O$ when $\mathbb{F}=\mathbb{R}$ or $\mathbb{C}$.
It is clear that
$$
\kappa:=[\Gamma_Q:\Gamma_Q^0]=
\begin{cases}
  4&\text{if $\mathbb{F}=\mathbb{R}$,}\\
  w&\text{if $\mathbb{F}=\mathbb{C}$,}\\
  1&\text{if $\mathbb{F}=\mathbb{H}$.}
\end{cases}
$$
Let $\{g_1,\dots,g_\kappa\}$ be a set of representatives of the $\Gamma_Q^0$-coclases of $\Gamma_Q$.
For example
$\left\{\left(\begin{smallmatrix}I_{n-1}&&\\& \pm1&\\&& \pm1\end{smallmatrix}\right)\right\}$ if $\mathbb{F}=\mathbb{R}$,
$\{\left(\begin{smallmatrix}I_{n}&\\ & \alpha\end{smallmatrix}\right):\alpha\in\mathcal O^\times\}$ if $\mathbb{F}=\mathbb{C}$ and
$\{I_{n+1}\}$ if $\mathbb{F}=\mathbb{H}$.

Now we can apply the lattice point theorem to our problem. For each $1\leq j\leq\kappa$, Theorem~\ref{thm:LPT_BMW} implies that
\begin{equation}\label{eq:g-in-GammaQ0}
\#\{g\in\Gamma_Q^0:d([e_{n+1}],g\cdot(g_j\cdot[y]))\leq s\}=
\frac{2^{1-n} \, m_{\Gamma_Q^0} \, \zeta}{2\rho\,\mathrm{vol}(\Gamma_Q^0\backslash \mathrm H_\mathbb{F}^n(Q))}\, e^{2\rho s} + O(e^{s\tau}).
\end{equation}
A trivial verification shows that $\mathrm{vol}(\Gamma_Q^0\backslash\mathrm H_\mathbb{F}^n(Q))= \xi_\mathbb{F}\, \mathrm{vol}(\Gamma_Q\backslash\mathrm H_\mathbb{F}^n(Q))$, with $\xi_\mathbb{R}=2$, $\xi_\mathbb{C}= m_{\Gamma_Q^0}$ and $\xi_{\mathbb H}=1$.
Furthermore, Lemma~\ref{lem:zeta} gives $\zeta=\mathrm{vol}(S^{nr-1})$.
These considerations imply, by adding formula \eqref{eq:g-in-GammaQ0} over $j$, that
\begin{equation}
\#\big\{g\in\Gamma_Q:d([e_{n+1}],g\cdot[y])\leq s\big\}\;=\;
\widetilde w\;
\displaystyle \frac{2^{1-n}\,\mathrm{vol}(S^{nr-1})}{2\rho\,\mathrm{vol}(\Gamma_Q\backslash \mathrm H_\mathbb{F}^n(Q))}\;e^{2\rho s}
    + O(e^{s\tau})
\end{equation}
where $\widetilde w=w$ if $\mathbb{F}=\mathbb{R},\mathbb{C}$ and $\widetilde w=2$ if $\mathbb{F}=\mathbb H$.
Applying this formula to \eqref{eq:main2}, we obtain that
\begin{equation}\label{eq:NafterLPT}
  N_t(Q,-k)=\widetilde w\;\frac{2^{1-n}\, \mathrm{vol}(S^{nr-1})}{2\rho\,\mathrm{vol}(\Gamma_Q\backslash \mathrm H_\mathbb{F}^n(Q))}\,\Big(\sum_{y\in F}|\Gamma_{Q,y}|^{-1}\Big)\, e^{2\rho s} + O(e^{s\tau}).
\end{equation}

Recall that $\cosh(s)=a^{1/2}\,k^{-1/2}\, t$.
Notice that we can replace the error term in \eqref{eq:NafterLPT} by $O(t^\tau)$ since $e^s\sim2\cosh(s)=2\,a^{1/2}\,k^{-1/2}\, t$ as $s\to+\infty$, and furthermore $e^{2\rho s}=2^{2\rho} a^{\rho} k^{-\rho}\, t^{2\rho} + O(t^\tau)$.

Collecting all the information in this section, we have obtained the following formula.
\begin{proposition}\label{prop:main3}
  The number $N_t(Q,-k)$ satisfies the asymptotic estimate
  \begin{equation}\label{eq:main3}
    N_t(Q,-k)= \widetilde w\;\frac{2^{2\rho-(n-1)}a^\rho\,\mathrm{vol}(S^{nr-1})}{2\rho\,k^\rho\,\mathrm{vol}(\Gamma_Q\backslash \mathrm H_\mathbb{F}^n(Q))}\,\Big(\sum_{y\in F}|\Gamma_{Q,y}|^{-1}\Big)\; t^{2\rho} + O(t^{\tau}),
  \end{equation}
as $t\to+\infty$, where $\tau$ is as in \eqref{eq:tau} and $\widetilde w=w$ if $\mathbb{F}=\mathbb{R},\mathbb{C}$ and $\widetilde w=2$ if $\mathbb{F}=\mathbb H$.
Moreover, when $\mathbb{F}=\mathbb{R}$ and $n>2$, \eqref{eq:main3} holds with $\tau=n-3/2$.
\end{proposition}
The last assertion follows from Remark~\ref{rmk:lambda_1}.

\section{The mass of the representation}
The object of this section is to obtain a formula for the term $\sum_{y\in F}|\Gamma_{Q,y}|^{-1}$ by using Siegel's theory on quadratic forms and its generalization to complex hermitian forms given by Raghavan.
Our main references are \cite{SiegelTata}, \cite{Ramanathan} and \cite{Raghavan}.
From now on we make the assumption $\mathbb{F}=\mathbb{R}$ or $\mathbb{C}$.
We denote by $d_{\mathcal O}$ the discriminant of the quotient field of $\mathcal O$.
The following Remark introduces a canonical volume element in algebraic varieties that will be useful (see \cite[\S5 Ch IV]{SiegelTata} for more details).

\begin{remark}\label{rmk:volume_element}
Let $x_1,\dots,x_m$ denote the coordinates of $\mathbb{R}^m$. Let $y_j=f_j(x_1,\dots,x_m)$ ($1\leq j\leq n$) be $n$ smooth functions with $n\leq m$.
Let $a_1,\dots,a_n\in\mathbb{R}$ such that the surface $\Omega=\{x\in\mathbb{R}^m: y_j(x)=a_j,\;\text{for } 1\leq j\leq n\}$ is non-singular, that is, the matrix with entries $\frac{\partial f_i}{\partial x_j}$ has maximum rank $n$ at every point of $\Omega$.
Choose $m-n$ differentiable functions $y_{n+1},\dots,y_m$ of $\mathbb{R}^m$ so that the Jacobian $J=\det(\frac{\partial y_i}{\partial x_j})$ is nonzero at every point of $\Omega$.
Then
\begin{equation}
  d\omega = |J|^{-1}\, d y_{n+1}\dots dy_m
\end{equation}
gives a volume element on $\Omega$ independently of the choice of $y_{n+1},\dots,y_m$. Siegel denoted this volume element by $\frac{\{dx\}}{\{dy\}}$.
\end{remark}

We pick $v\in\mathbb{F}^n$ and $R\in\mathrm M(n,\mathbb{F})$ such that the matrix
$$
W=\begin{pmatrix}
-k&v\\v^t&R
\end{pmatrix}
$$
is $\mathbb{F}$-hermitian of signature $(n,1)$.
For $y\in\mathbb{F}^{n+1}$ such that $Q[y]=-k$, let $\mathrm{U}(Q,\mathbb{F})_y$ denote the set of elements $U\in\mathrm{U}(Q,\mathbb{F})$ such that $Uy=y$.
Note that if $y\in\mathcal O^{n+1}$, the stabilizer of $y$ in $\Gamma_Q$ is $\Gamma_{Q,y}=\mathrm{U}(Q,\mathbb{F})_y\cap \mathrm{GL}(n+1,\mathcal O)$.
Consider the varieties
\begin{eqnarray*}
\Omega(Q,W)&=&\{X\in\mathrm{M}(n+1,\mathbb{F}) : Q[X]=W\},\\
\Omega(Q,W;y)&=&\{Y\in\mathrm{M}(n+1,n;\mathbb{F}) : Q[(y\,|\,Y)]=W\}.
\end{eqnarray*}
The groups $\mathrm{U}(Q,\mathbb{F})$ and $\mathrm{U}(Q,\mathbb{F})_y$ act by left multiplication on $\Omega(Q,W)$ and on $\Omega(Q,W;y)$ respectively. On these varieties we fix the volume elements
$$
 d\omega =
\left|\frac{\det(W)}{\det(Q)}\right|^{\frac{2-r}{2}}\,
\frac{\{ d  X\}}{\{ d  W\}}
\qquad\text{and}\qquad
 d\omega^* =
\left|\frac{\det(W)}{\det(Q)}\right|^{\frac{2-r}{2}}\,
\frac{\{ d  X\}}{\{ d  v\}\{ d  R\}},
$$
respectively, where $\frac{\{ d  X\}}{\{ d  W\}}$ and $\frac{\{ d  X\}}{\{ d  v\}\{ d  R\}}$ are the volume elements given in Remark~\ref{rmk:volume_element}.
In the complex case, we write $X=X^{(1)}+iX^{(2)}$ and $W=W^{(1)} + iW^{(2)}$ with $X^{(1)}$, $X^{(2)}$, $W^{(1)}$ and $W^{(2)}$ real matrices. Thus, the volume element $\frac{\{ d  X\}}{\{ d  W\}}$ is defined by considering the algebraic equations $\mathrm{Re}(Q[X^{(1)}+iX^{(2)}])=W^{(1)}$ and $\mathrm{Im}(Q[X^{(1)}+iX^{(2)}])=W^{(2)}$.
The factor $\left|{\det(W)}/{\det(Q)}\right|^{\frac{2-r}{2}}$ is included so that $d\omega$ does not depend on $W$.
Similar consideration apply for $\frac{\{ d  X\}}{\{ d  v\}\{ d  R\}}$ and $d\omega^*$.

Siegel (and \cite{Ramanathan} for the hermitian case) defines \emph{the measure of the representation of $-k\in\mathbb{Z}$ by $Q$} as
\begin{equation}\label{eq:measurerep}
\mu(Q,-k)=\sum_{y\in F} \mu(y,Q)/ \mu(Q).
\end{equation}
Here $\mu(Q)$ denotes the \emph{measure of the unit group $\Gamma_Q$} given by $(r^2/|d_{\mathcal O}|)^{(n+1)(n+2)/4}$ times the volume of any fundamental domain for the action of $\Gamma_Q$ on $\Omega(Q,W)$, and similarly,
$\mu(y,Q)$ is the \emph{measure of the representation $y$}, given by $(r^2/|d_{\mathcal O}|)^{n(n+1)/4}$ times the volume of any fundamental domain for the action of $\Gamma_{Q,y}$ on $\Omega(Q,W;y)$.

We will recover from the right hand side of \eqref{eq:measurerep} the term $\sum_{y\in F}|\Gamma_{Q,y}|^{-1}$ and then, by applying Siegel's main theorem, we will obtain an explicit formula for this term.
By \cite[Thm.~7, Ch~IV]{SiegelTata} and \cite[(93)]{Raghavan} (or \cite[(70)]{Ramanathan}) we have that
\begin{equation}\label{eq:volS}
w\,\mu(Q)
=\big(r^2/{|d_{\mathcal O}|}\big)^{\frac{(n+1)(n+2)}4}\,
\frac{\mathrm{vol}(\Gamma_Q\backslash\mathrm H_\mathbb{F}^n(Q))}{r^{n+1}\, |\det Q|^{\frac{r}{2}n+1} }\,
\frac{\pi^{\frac{r}{2}}}{\Gamma(\frac{r}{2})}
\prod_{j=1}^n\frac{\pi^{\frac{r}{2}j}}{\Gamma(\frac{r}{2}j )},
\end{equation}
where $w=\#\mathcal O^\times$.

Let $\widetilde{\mathcal F}(y)$ be a fundamental domain of the action of $\Gamma_{Q,y}$ on $\Omega(Q,W;y)$.
By definition
$
\mu(y,Q)=
|\det W|^{(2-r)/2}\,|\det Q|^{-(2-r)/2}
(r^2/{|d_{\mathcal O}|})^{n(n+1)/4} \int_{\widetilde{\mathcal F}(y)}  d \omega^*,
$
but the measure $ d \omega^*$ is invariant by $\Gamma_{Q,y}$, then
$$
\mu(y,Q)=
 \frac{1}{|\Gamma_{Q,y}|}\,
 \left|\frac{\det W}{\det Q}\right|^{(2-r)/2}
 \big({r^2}/{{|d_{\mathcal O}|}}\big)^{\frac{n(n+1)}{4}}
 \int_{\Omega(Q,W;y)}  d \omega^*.
$$
Using \cite[Thm.~6, Ch~IV]{SiegelTata} and \cite[Lemma~9]{Ramanathan}, we have that
\begin{equation}\label{eq:volS*}
\mu(y,Q)=
|\Gamma_{Q,y}|^{-1}\,
\big({r^2}/{{|d_{\mathcal O}|}}\big)^{\frac{n(n+1)}{4}}\,
\frac{k^{1-\frac{r}{2}(n+1)}}{|\det Q|^{\frac{r}{2}(n-1)+1}}\;
\prod_{j=1}^n\frac{\pi^{\frac{r}{2} j }}{\Gamma(\frac{r}{2} j )}.
\end{equation}

Finally, \eqref{eq:volS} and \eqref{eq:volS*} imply that
\begin{equation}\label{eq:mu(Q,-k)}
  \mu(Q,-k) =
  w\,
  |d_{\mathcal O}|^{\frac{n+1}2}
  \frac{k^{1-\frac{r}{2}(n+1)} |\det Q|^{\frac{r}{2}}}
  {\mathrm{vol}(\Gamma_Q\backslash\mathrm H_\mathbb{F}^n(Q))}
  \frac{\Gamma(\frac{r}{2})}{\pi^{\frac{r}{2}}}
  \sum_{y\in F} |\Gamma_{Q,y}|^{-1}.
\end{equation}

Now, we will recall Siegel's main theorem for indefinite quadratic and hermitian forms (see \cite[Thm.~1]{SiegelIndefinite} and \cite[Thm.~7]{Raghavan}).
For every rational prime $p$, the \emph{$p$-adic density of representation of $-k$ by $Q$} is
\begin{equation}\label{eq:densidad}
\delta_p(Q,-k)=
\lim_{j\to\infty} p^{-j(r (n+1)-1)}\;\# \{x\in{\big(\mathcal O /p^j\mathcal O\big)}^{n+1}: Q[x]\equiv -k\;\bmod p^j \}.
\end{equation}
Define $\delta(Q,-k)=\prod_{p} \delta_p(Q,-k)$, the \emph{local density}, where the product is over all prime numbers.

\begin{theorem}[\emph{Siegel's mass formula}]\label{thm:Siegel'sMass}
Let $Q$ be an $\mathbb{F}$-hermitian form as in \eqref{eq:Q} with $n\geq2$ and let $k$ be a positive integer.
Then
\begin{equation}\label{eq:Siegel'sMass}
\mu(Q,-k)=\delta(Q,-k).
\end{equation}
\end{theorem}

See \cite[Thm.~1]{SiegelIndefinite} for the real case and \cite[Thm.~7]{Raghavan} for the complex case.

By combining equation \eqref{eq:mu(Q,-k)} and Theorem~\ref{thm:Siegel'sMass}, we obtain an expression for the term $\sum_{y\in F} |\Gamma_{Q,y}|^{-1}$, the main goal of this section.

\begin{corollary}\label{cor:term}
We have that
$$
\sum_{y\in F} |\Gamma_{Q,y}|^{-1}=
w^{-1}
|d_{\mathcal O}|^{-\frac{n+1}2}
\frac{k^{\frac{r}{2}(n+1)-1}}{|\det Q|^{\frac{r}{2}}}
\mathrm{vol}(\Gamma_Q\backslash\mathrm H_\mathbb{F}^n(Q))\;
\frac{\pi^{\frac{r}{2}}}{\Gamma(\frac{r}{2})}
\delta(Q,-k)
$$
\end{corollary}
Note that ${\pi^{r/2}}/{\Gamma(r/2)}=1,\pi,\pi^2$ for $\mathbb{F}=\mathbb{R},\mathbb{C}$ and $\mathbb H$ respectively.

\section{Main theorem}

We can now state the main result in this paper, which follows by combining Proposition~\ref{prop:main3} and Corollary~\ref{cor:term}. The last assertion is a consequence of Remark~\ref{rmk:lambda_1}.
We first recall some terminology: $r=\dim_\mathbb{R}(\mathbb{F})$, $\rho=(n+1)\,r/2-1$ and $\mathcal O$ is a maximal order in $\mathbb{F}$ with discriminant $d_{\mathcal O}$.

\begin{theorem}\label{thm:main4}
Let $Q$ be an $\mathbb{F}$-hermitian matrix as in \eqref{eq:Q}, with $n\geq2$ and $\mathbb{F}=\mathbb{R}$ or $\mathbb{C}$.
We fix $k\in\mathbb{N}$ such that $-k$ is represented by $Q$.
Then, the number $N_t(Q,-k)$ of elements $x\in\mathcal O^{n+1}$ such that $Q[x]=-k$ and $|x_{n+1}|\leq t$, satisfies the following asymptotic estimate as $t\to+\infty$,
\begin{equation}\label{eq:main4}
    N_t(Q,-k)=
    \frac{2^{(r-1)(n+1)}}{|d_{\mathcal O}|^{\frac{n+1}2}}
    \frac{a^\rho\,\mathrm{vol}(S^{nr-1})}{2\rho\,|\det Q|^{\frac{r}{2}}}
    \frac{\pi^{\frac{r}{2}}}{\Gamma(\frac{r}{2})}\;
    \delta(Q,-k)\;
    t^{2\rho} + O(t^{\tau}),
\end{equation}
where $\tau$ is as in \eqref{eq:tau}.
Moreover, when $\mathbb{F}=\mathbb{R}$ and $n>2$, the formula \eqref{eq:main3} holds for $\tau=n-3/2$.
\end{theorem}

\begin{remark}
In order to get an explicit value of the main term in \eqref{eq:main4} for a fixed $\mathbb{F}$-hermitian form $Q$, one needs to determine $\delta(Q,-k)$.
When $\mathbb{F}=\mathbb{R}$, T.~Yang \cite{Yang} computed this local density for any quadratic form $Q$. For $\mathbb{F}=\mathbb{C}$ see \cite{Hironaka}.
\end{remark}

\begin{remark}\label{rmk:norms}
We consider the real norm $\|\cdot\|$ induced by the $\mathbb{F}$-hermitian form $\left(\begin{smallmatrix}A&\\&a\end{smallmatrix}\right)$,
i.e.\ $\|x\|^2 = A[\hat x]+a|x_{n+1}|^2$, where $\hat x=(x_1,\dots,x_n)^t$.
We claim that Theorem~\ref{thm:main4} provides an asymptotic formula with error term for $\widetilde N_s(Q,-k)$ (see~\eqref{eq:intro-N_s(Q,k)})
as $s\to+\infty$, of the number of solutions lying in the Euclidean ball of radius $s$.
Indeed, it is clear that
$$
\left.
\begin{array}{r@{\;}c@{\;}l}
  Q[x]    &=& A[\hat x]-a|x_{n+1}|^2 = -k \\[1mm]
  \|x\|^2 &=& A[\hat x]+a|x_{n+1}|^2 \leq s^2
\end{array}
\right\}
\quad \text{ imply }\quad
|x_{n+1}| \leq \sqrt{\frac{s^2+k}{2}}=:t_s.
$$
Hence, by Theorem~\ref{thm:main4}, we have
\begin{equation}\label{eq:tilde_N-1st}
\widetilde N_s(Q,-k) = N_{t_s}(Q,-k) = C_{Q,k}\; t_s^{2\rho} + O(t_s^{\tau}),
\end{equation}
where $C_{Q,k}$ denotes the main coefficient of \eqref{eq:main4}.
By Taylor's expansion at $s=\infty$, $t_s = s/\sqrt{2} +O(s^{-1})$,
thus we can replace $O(t_s^{\tau})$ by $O(s^{\tau})$ in \eqref{eq:tilde_N-1st}.
Moreover, $t_s^{2\rho} = s^{2\rho}/2^{\rho}+O(s^{2\rho-2})$.
But $\tau \geq 2\rho\, n/(n+1)+\varepsilon$ by \eqref{eq:tau}, and so it follows immediately that $\tau-(2\rho-2)>0$ for $\mathbb{F}=\mathbb{R}$ and $\mathbb{C}$.
Therefore, \eqref{eq:tilde_N-1st} reduces to
\begin{equation}
\widetilde N_s(Q,-k) = 2^{-\rho}\, C_{Q,k} \; s^{2\rho} + O(s^{\tau}).
\end{equation}
\end{remark}

\begin{example}
The case $\mathbb{F}=\mathbb{R}$ and
$
Q=I_{n,1}=\left(\begin{smallmatrix}I_n&\\&-1\end{smallmatrix}\right)
$
was considered by J.~Ratcliffe and S.~Tschantz \cite{Ratcliffe-Tschantz}.
We have $r=1$, $|d_{\mathcal O}|=1$, $a=1$, $|\det(I_{n,1})|=1$ and $\rho=(n-1)/2$. Theorem~\ref{thm:main4} now yields
\begin{equation}
    N_t(I_{n,1},-k)=
    \frac{\mathrm{vol}(S^{n-1})}{n-1}\;
    \delta(I_{n,1},-k)\;
    t^{n-1} + O(t^{n-3/2}).
\end{equation}
In \cite[Thm.~12]{Ratcliffe-Tschantz} there is an explicit formula for the local density $\delta(I_{n,1},-k)$.
\end{example}

\begin{example}
We conclude the article by considering the case of the Lorentzian hermitian form over the Gaussian integers, i.e.\ $Q=I_{n,1}$ for $\mathbb{F}=\mathbb{C}$ and $\mathcal O=\mathbb{Z}[\sqrt{-1}]$.
We have $r=2$, $\rho=n$, $d_{\mathcal O}=-4$, $a=1$ and $|\det(I_{n,1})|=1$.
Furthermore, $\mathrm{vol}(S^{2n-1})=2\pi^{n}/(n-1)!$.
Theorem~\ref{thm:main4} now implies that
\begin{align}\label{eq:Lorentzianhermitian}
    N_t(I_{n,1},-k) = \dfrac{\pi^{n+1}}{n!}\; \delta(I_{n,1},-k)\; t^{2n} + O(t^{\tau}).
\end{align}
The local density $\delta(I_{n,1},-k)$ is explicitly computable.
For example, when $n=2$ and $k=1$ we have $\delta(I_{2,1},-1)=2^33\pi^{-3}$ and consequently $N_t(I_{2,1},-1) = 12\, t^{4} + O(t^{\tau}).$
In a future paper we will compute this term for several examples and testing our formula with experimental computations.
\end{example}

\section*{Acknowledgments}
This is part of the author's Ph.D.\ thesis, written under the supervision of Professor Roberto Miatello at the Universidad Nacional de C\'ordoba, Argentina.
The author sincerely thanks Roberto for suggesting the problem and for many helpful discussions concerning the material in this paper.
The author also wishes to express his thanks
 to Wai Kiu Chan, Rainer Schulze-Pillot and Takao Watanabe for several helpful comments and
 to Jorge Lauret and the referee for reading the draft and making helpful suggestions.

\end{document}